\documentclass[a4paper,10pt]{article}
\usepackage[a4paper, total={6in, 8in}]{geometry}

\usepackage{amsmath}
\usepackage{amsfonts}
\usepackage{amsthm}
\usepackage{amssymb}
\usepackage{enumitem}
\usepackage{yhmath}
\usepackage[all]{xy}
\usepackage{marginnote}
\usepackage{pgf,tikz,pgfplots}
\usepackage{mathrsfs}
\usetikzlibrary{arrows}
\numberwithin{equation}{section}

\newcommand{\C}{\mathbb{C}}
\newcommand{\R}{\mathbb{R}}
\newcommand{\N}{\mathbb{N}}
\newtheorem{theorem}{Theorem}[section]
\newtheorem{lemme}[theorem]{Lemma}

\newtheorem{prop}[theorem]{Proposition}

\newtheorem{example}[theorem]{Example}
\newcommand{\RC}{RC_{\lambda',\lambda''}^{\lambda'''}}
\newcommand{\Jacobi}{P_l^{\lambda'-1,\lambda''-1}}

\author{Quentin Labriet\footnote{Laboratoire de Mathématiques de Reims (LMR-UMR 9008 du CNRS),
Université de Reims Champagne-Ardenne Moulin de la Housse - BP 1039
51687 REIMS cedex 2, France. Email address: quentin.labriet@univ-reims.fr}}
\title{Holographic transform for tensor product of holomorphic discrete series}
\date{}
\begin{document}

\maketitle


\begin{abstract}
We study holographic operators associated with Rankin-Cohen brackets which are symmetry breaking operators for the restriction of tensor products of holomorphic discrete series of the universal covering of $SL_2(\R)$. Furthermore, we investigate a geometrical interpretation of these operators and their relations to classical Jacobi polynomials.
\end{abstract}

\section{Introduction}

Let $G$ be a real reductive Lie group, $G'$ a Lie subgroup of $G$ and $\pi$ an irreducible unitary representation of $G$ on a vector space $V$. The decomposition into irreducible representations $(\rho,W)$ of $G'$ of the restriction $\pi|_{G'}$ of $\pi$ to $G'$ is called the branching rule. \emph{Symmetry breaking operators} are defined as  elements of the space of linear continuous maps $V\to W$ that intertwine $\pi|_{G'}$ and $\rho$, denoted by $\operatorname{Hom}_{G'}(\pi|_{G'},\rho)$, for any given  irreducible representation $(\rho,W)$ of $G'$. Analogously, elements of the space $\operatorname{Hom}_{G'}(\rho,\pi|_{G'})$ are called \emph{holographic operators} (see for example \eqref{eq:Adjoint_RC_reproducing}). 

In the article \cite{KP}, the authors investigate holographic operators in two different geometric settings : $(G,G')=(\widetilde{SL_2(\R)}\times \widetilde{SL_2(\R)},\widetilde{SL_2(\R)})$ referred to as the \emph{diagonal case}, and $(G,G')=(SO_0(2,n),SO_0(2,n-1))$ referred to as the \emph{conformal case}. In both situations, they consider the restriction of holomorphic discrete series representations of $G$ to $G'$ and develop two explicit approaches to study corresponding holographic operators:
\begin{itemize}
\item The first method is based on the Laplace transform for tube domains, and leads to a new transform involving integration along a line segment in the complex upper half-plane $\Pi$ (see \ref{eq:Operator_holographic_holomorphe}) for the diagonal case.
\item The second one uses the reproducing kernel technics for weighted Bergman spaces, and leads to the construction of a \emph{relative reproducing kernel} in the conformal case (see Thm 3.10 in \cite{KP}).
\end{itemize}
It is known (see \cite{Dijk_Pev} for instance) that the symmetry breaking operators for the decomposition of the tensor products of holomorphic discrete series of $SL_2(\R)$ (what we call the diagonal case) are proportional to Rankin-Cohen brackets $\RC$ (see \eqref{eq:RC}).

In this work, we extend the method of the relative reproducing kernel to the diagonal case and describe the corresponding holographic $(\RC)^*$ operator as follows.

\begin{theorem}\label{thm:Fomule_adjoint_RC_relative_reproducing}
Suppose $\lambda',\lambda'',\lambda'''>1$ such that $l=\frac{1}{2}(\lambda'''-\lambda'-\lambda'')\in \N$. Let $w_1, w_2 \in \Pi$, and $g\in H^2_{\lambda'''}(\Pi)$. Then we have:
\begin{equation}
(\RC)^*g(w_1,w_2)=C(\lambda',\lambda'') \int_\Pi g(z) K_{\lambda',\lambda''}^{\lambda'''}(z,w_1,w_2)~d\mu(z),
\end{equation}
where $K_{\lambda',\lambda''}^{\lambda'''}(z,w_1,w_2)=(w_2-w_1)^l\left(\frac{w_1-\bar{z}}{2i}\right)^{-(\lambda'+l)}\left(\frac{w_2-\bar{z}}{2i}\right)^{-(\lambda''+l)}$, $d\mu(x+iy)=y^{\lambda'''-2}~dxdy$, and $C(\lambda',\lambda'')=\frac{(\lambda'-1)_{l+1}(\lambda''-1)_{l+1}}{2^{2l+4}\pi^2 l!}$.
\end{theorem}
Notice that this kind of operators has already been considered by H. Rosengren in \cite{Rosen}. We give two different proofs of this theorem, the first one is inspired by the proof of Theorem 3.10 in \cite{KP}, and the second one, based on the Laplace transform, is new and can also be used in the conformal case.\\

Our second point (see Fact.\ref{fact:sym_break}) explores a conceptual interpretation of the link between orthogonal polynomials and branching rules for restriction of discrete series representations. More precisely, Rankin-Cohen brackets can be expressed in terms of Jacobi polynomials as showed in \cite{KP_SBO} (see Thm.8.1). This is due to the fact that the equivariance condition for symmetry breaking operators can be reduced, using the F-method  (see \cite{KP_SBO_1}), to the Jacobi ordinary differential equation for the symbols of such operators. In \cite{KP} p.15, the authors give yet another interpretation of the link between Rankin-Cohen brackets and Jacobi polynomials based on the fact that holographic transform can be expressed using the inversion of the usual Jacobi transform. In this work, we interpret this statement emphasising the fact that Jacobi polynomials form an orthogonal basis for the Hilbert space $L^2((-1,1),(1-v)^{\alpha}(1+v)^{\beta}~dv)$ (see \cite{KP}, section 5.1). For this, we use another realization of the holomorphic discrete series in which symmetry breaking operators are expressed in terms of the classical Jacobi transform. \\
\\
Notation: We use the Pochammer symbol, defined for $\lambda \in \C$ and $n\in\N$ by :
\[(\lambda)_n=\lambda(\lambda+1)\cdots(\lambda+n-1)=\frac{\Gamma(\lambda+n)}{\Gamma(\lambda)}.\]

and set the condition
\begin{equation} \label{condition}
\lambda',\lambda'',\lambda'''>1 \text{ such that }l:=\frac{1}{2}(\lambda'''-\lambda'-\lambda'')\in \N.
\tag{C1}
\end{equation}

\section{Setting}\label{sec:Setting}
In this section, we describe two different models for the holomorphic discrete series representations of the universal covering group of the Lie group $SL_2(\R )$, denoted $\widetilde{SL_2(\R)}$, and their tensor products. 
\subsection{Holomorphic model}
We denote by $\Pi$ the Poincaré upper half-plane endowed with the hyperbolic metric, and by $H^2_\lambda(\Pi)$ the weighted Bergman space defined by:
\[H^2_\lambda(\Pi)=\mathcal{O}\left(\Pi\right)\cap L^2\left(\Pi,y^{\lambda-2}~dxdy\right).\]
It is known that $H^2_\lambda(\Pi)=\{0\}$ for $\lambda \leq 1$, so we suppose that $\lambda>1$. In this case $H^2_\lambda(\Pi)$ admits a reproducing kernel, also called the Bergman kernel, given  by (see, for instance, \cite{Far_Kor}, Prop.XIII.1.2, p.261):

 \begin{equation} \label{eq:Noyau_bergmann}
 K_\lambda(z,w)=\frac{\lambda-1}{4\pi}\left(\frac{z-\bar{w}}{2i}\right)^{-\lambda}.
 \end{equation}
We set, for $\lambda',\lambda''>1$, $H^2_{\lambda',\lambda''}(\Pi \times \Pi)\simeq H^2_{\lambda'}(\Pi)\hat{\otimes}H^2_{\lambda''}(\Pi)$, where $\hat{\otimes}$ denotes the Hilbert space completion of the tensor product. This space also admits a reproducing kernel:

\begin{equation}\label{eq:Noyau_produit_Bergman}
K_{\lambda',\lambda''}(z,w)=K_{\lambda'}(z_1,w_1)\cdot K_{\lambda''}(z_2,w_2) =\frac{(\lambda'-1)(\lambda''-1)}{(4\pi)^2}\left(\frac{z_1-\bar{w_1}}{2i}\right)^{-\lambda'}\left(\frac{z_2-\bar{w_2}}{2i}\right)^{-\lambda''}.
\end{equation}

For $\lambda\in \N\backslash \{0;1\}$, the holomorphic discrete series representations $\pi_\lambda$ of $SL_2(\R)$ can be realized on $H_\lambda^2(\Pi)$ by the following formula:
\begin{equation}\label{eq:holomorphic_discrete_series}
 \left(\pi_{\lambda}(g)f\right)(z)=(cz+d)^{-\lambda}f\left(\frac{az+b}{cz+d}\right),
\end{equation}
where $g^{-1}=\begin{pmatrix}
a&b\\c&d
\end{pmatrix} \in SL_2(\R),$ and $f \in H_\lambda^2(\Pi)$.
This representation lifts to a unitary and irreducible representation of the universal covering group $\widetilde{SL_2(\R)}$ for $\lambda>1$, through an appropriate choice of the determination of the power function.

The outer product representation $\pi_{\lambda'}\boxtimes\pi_{\lambda''}$ of the direct product group $\widetilde{SL_2(\R)}\times \widetilde{SL_2(\R)}$ acts on the space $H^2_{\lambda',\lambda''}(\Pi\times\Pi)$. This representation is also irreducible and unitary for $\lambda',\lambda''>1$.

\subsection{Rankin-Cohen operators}
Let $\lambda',\lambda'',\lambda'''$ verify \eqref{condition}, and define a differential operator $R_{\lambda',\lambda''}^{\lambda'''}$ from $\mathcal{O}(\Pi\times \Pi)$ to $\mathcal{O}(\Pi\times \Pi)$ by
\begin{equation}\label{eq:RC}
R_{\lambda',\lambda''}^{\lambda'''}(f)(w_1,w_2)=\sum_{j=0}^l \frac{(-1)^j(\lambda'+l-j)_j(\lambda''+j)_{l-j}}{j!(l-j)!}\frac{\partial^lf}{\partial z_1^{l-j} \partial z_2^j}(w_1,w_2).
\end{equation}

The \emph{Rankin-Cohen operator} is a map from $\mathcal{O}(\Pi\times \Pi)$ to $\mathcal{O}(\Pi)$ defined by :
\begin{equation}
\RC = Rest \circ R_{\lambda',\lambda''}^{\lambda'''},
\end{equation}
where  $Rest$ is the restriction to the diagonal of $\Pi\times\Pi$.
The Rankin-Cohen operator $\RC$ is a symmetry breaking operator for the outer product representation $\pi_{\lambda'}\boxtimes\pi_{\lambda''}$, more precisely it generates the space $\operatorname{Hom}_{\widetilde{SL_2(\R)}}(\pi_{\lambda'}\boxtimes\pi_{\lambda''}|_{\widetilde{SL_2(\R)}},\pi_{\lambda'''})$, for $\lambda',\lambda'',\lambda'''$ satisfiying \eqref{condition} (see \cite{KP_SBO} Cor. 9.3 for the precise statement). \\

The goal of this paper is to study the holographic operators associated to the Rankin-Cohen operators. In order to do so, we compute the adjoint operators $(RC_{\lambda',\lambda''}^{\lambda'''})^*: H^2_{\lambda'''}(\Pi)\to H^2_{\lambda',\lambda''}(\Pi\times\Pi)$ which give, in the unitary case, the holographic operators.

\subsection{Holographic operators for the diagonal case}
Let $\lambda',\ \lambda'',\ \lambda'''$ verify \eqref{condition}, and define the operators $\Psi_{\lambda',\lambda''}^{\lambda'''}~:~H^2_{\lambda'''}(\Pi)\to H_{\lambda',\lambda''}^2(\Pi\times\Pi)$ by 
\begin{equation}\label{eq:Operator_holographic_holomorphe}
\Psi_{\lambda',\lambda''}^{\lambda'''}(g)(w_1,w_2)=\frac{(w_1-w_2)^l}{2^{\lambda'+\lambda''+2l-1}l!}\int_{-1}^1g(w(v))(1-v)^{\lambda'+l-1}(1+v)^{\lambda''+l-1}dv,
\end{equation}
where $w(v)=\frac{1}{2}((w_2-w_1)v+(w_2+w_1))$.\\
In \cite{KP}, the authors prove that this is a holographic operator from $H^2_{\lambda'''}(\Pi)$ to $H^2_{\lambda',\lambda''}(\Pi\times \Pi)$.
They also establish a Parseval-Plancherel type theorem for the corresponding Rankin-Cohen transorm and its associated holographic transform (see \cite{KP} Thm. 2.7).
Finally, they show (Prop 2.23) that :
\begin{equation}\label{eq:link_Psi_dualRC}
(\RC)^*=C\Psi_{\lambda',\lambda''}^{\lambda'''}.
\end{equation}
where $C=\frac{\Gamma(\lambda'+\lambda''+2l-1)}{2^{2l+2}\pi\Gamma(\lambda'-1)\Gamma(\lambda''-1)}$.

However, they also used another method to get a different expression for this map (see Lemma \ref{lem:calc_adj} below). This approach leads to a different type of integral transformations based on the relative reproducing kernel $K_{\lambda',\lambda''}^{\lambda'''}$, defined for $w_2, w_1, z \in \Pi$ by:
\begin{equation}\label{eq:Relative_reproducing}
K_{\lambda',\lambda''}^{\lambda'''}(z,w_1,w_2)=(w_2-w_1)^l\left(\frac{w_1-\bar{z}}{2i}\right)^{-(\lambda'+l)}\left(\frac{w_2-\bar{z}}{2i}\right)^{-(\lambda''+l)}.
\end{equation}

This relative reproducing kernel gives an explicit expression for the adjoint of the Rankin-Cohen operator as follows
\begin{theorem}
Set $\lambda',\lambda'',\lambda'''$ verify \eqref{condition}. Let $w_1,\ w_2 \in \Pi$, and $g\in H^2_{\lambda'''}(\Pi)$. Then we have:
\begin{equation}\label{eq:Adjoint_RC_reproducing}
(\RC)^*g(w_1,w_2)=C(\lambda',\lambda'') \int_\Pi g(z) K_{\lambda',\lambda''}^{\lambda'''}(z,w_1,w_2)d\mu(z),
\end{equation}
where $C(\lambda',\lambda'')=\frac{(\lambda'-1)_{l+1}(\lambda''-1)_{l+1}}{2^{2l+4}\pi^2 l!}$.
\end{theorem}

We prove this theorem in the following sections (\ref{sec:Premiere_preuve} and \ref{sec:Deuxieme_preuve}) and make explicit the link with the operator $\Psi_{\lambda',\lambda''}^{\lambda'''}$ introduced in (\ref{eq:Operator_holographic_holomorphe}). For this we need another model for the holomorphic discrete series representations of $\widetilde{SL_2(\R)}$.

\subsection{$L^2$-model of holomorphic discrete series}
For $\lambda>1$, the Laplace transform defined by:
\begin{equation}\label{eq:fourier_transform_SL2}
\mathcal{F} g(z) =\int_0^\infty g(t)e^{izt}dt,
\end{equation}
is a one-to-one isometry (up to a constant) from $L^2_\lambda(\mathbb{R}^+):=L^2(\mathbb{R}^+,t^{1-\lambda}dt)$ to $H^2_\lambda (\Pi)$ (see for instance \cite{Far_Kor}, Thm. XII.1.1). More precisely:
\[\|\mathcal{F} g\|^2_{H^2_\lambda(\Pi)}=b(\lambda)\|g\|^2_{L^2_{\lambda}(\R^+)},\] for every $g\in L^2_{\lambda}(\R^+)$, where $b(\lambda)=2^{2-\lambda}\pi\Gamma(\lambda-1)$.

Using this transform, we consider another realization for the holomorphic discrete series representations $\pi_\lambda$ of $\widetilde{SL_2(\R)}$ on $L^2_\lambda(\R^+)$, and call it the \emph{$L^2$-model} for $\pi_\lambda$.

For $\lambda',\lambda''>1$, we define
\[L^2_{\lambda',\lambda''}(\R^+\times\R^+):=L^2_{\lambda'}(\R^+)\hat{\otimes}L^2_{\lambda''}(\R^+)\simeq L^2(\R^+\times\R^+,x^{1-\lambda'}y^{1-\lambda''}dxdy).\]
We denote $\mathcal{F}_2=\mathcal{F}\otimes\mathcal{F}$ the Laplace transform from $L^2_{\lambda',\lambda''}(\R^+\times\R^+)$ to $H^2_{\lambda',\lambda''}(\Pi\times \Pi)$.\\

\subsection{Holographic operators in the $L^2$-model}
Set $\lambda',\lambda'',\lambda'''$ satisfiying \eqref{condition}. Using the Laplace transform (\ref{eq:fourier_transform_SL2}), we define the analogue of the Rankin-Cohen operators in the $L^2$-model by:
\begin{equation}\label{eq:RC_hat_def}
\widehat{\RC}:=\mathcal{F}^{-1}\circ \RC \circ \mathcal{F}_2.
\end{equation}
It is a symmetry breaking operator for the $\widetilde{SL_2(\R)}$-action in the $L^2$-model. 

In \cite{KP}, the authors prove that $\widehat{\RC}$ is given by the following integral formula for $F\in L^2_{\lambda',\lambda''}(\R^+ \times \R^+)$ (see \cite{KP}, Prop. 2.13):
\begin{equation}\label{eq:RC_hat}
\widehat{\RC} F(t)=\frac{t^{l+1}}{2i^l}\int_{-1}^1P_l^{\lambda'-1,\lambda''-1}(v)F\left(\frac{t}{2}(1-v),\frac{t}{2}(1+v)\right)~dv,
\end{equation}
where $\Jacobi$ denotes the Jacobi polynomials (see \cite{KP}, section 5.1).

We define the following operator which associates to a function $g(t)$ defined on $\R^+$ a function of two variables $\Phi_{\lambda',\lambda''}^{\lambda'''}g$ on $\R^+\times\R^+$ :
\begin{equation}\label{eq:holographic_L2_phi}
\Phi_{\lambda',\lambda''}^{\lambda'''}g(x,y):=\frac{x^{\lambda'-1}y^{\lambda''-1}}{(x+y)^{\lambda'+\lambda''+l-1}}P_l^{\lambda'-1,\lambda''-1}\left(\frac{y-x}{x+y}\right)\cdot g(x+y).
\end{equation}
Once again, it is shown in \cite{KP}, that $\Phi_{\lambda',\lambda''}^{\lambda'''}$ is a holographic operator between $L^2_{\lambda'''}(\R^+)$ and $L^2_{\lambda',\lambda''}(\R^+\times \R^+)$. 
We summarize the framework in the two following commutative diagrams

 \[\xymatrix{
    L^2_{\lambda',\lambda''}(\R^+\times\R^+) \ar[rrr]^{\widehat{\RC}}\ar[d]_{\mathcal{F}_2}  &&& L^2_{\lambda'''}(\R^+) \ar[d]^{\mathcal{F}} \\
    H^2_{\lambda',\lambda''}(\Pi\times\Pi)\ar[rrr]_{\RC} &&& H^2_{\lambda'''}(\Pi)
  }  
  \]
  \begin{center}
  Symmetry breaking operators in the holomorphic and $L^2$-models
  \end{center}
  \[
   \xymatrix{
    L^2_{\lambda',\lambda''}(\R^+\times \R^+) \ar[d]_{\mathcal{F}_2}  &&& L^2_{\lambda'''} (\R^+)\ar[d]^{\mathcal{F}} \ar[lll]_{\Phi_{\lambda',\lambda''}^{\lambda'''} }\\
    H^2_{\lambda',\lambda''} (\Pi\times\Pi) &&& H^2_{\lambda'''}(\Pi)\ar[lll]^{\Psi_{\lambda',\lambda''}^{\lambda'''}}
  }\]
\begin{center}  
  Holographic operators in the holomorphic and $L^2$-models
\end{center}

\subsection{Inverse Laplace transform of the reproducing kernel}

The following lemma gives the inverse Laplace transform for the Bergman kernels $K_\lambda(\cdot,w)$.

\begin{lemme}\label{cor:Inv_Fourier_SL2}
Suppose $\lambda,~\lambda',~\lambda''>1$. The inverse Laplace transform of the reproducing kernel $K_\lambda$ is given by the following fomula:
\begin{equation}
\mathcal{F}^{-1}K_\lambda(\cdot,z)(t)=\frac{2^{\lambda-1}}{2\pi\Gamma(\lambda-1)}t^{\lambda-1}e^{-it\bar{z}}.
\end{equation}
Consequently:
\begin{equation}
\mathcal{F}_2^{-1}K_{\lambda',\lambda''}(\cdot,(w_1,w_2))(x,y)=\frac{2^{\lambda'+\lambda''-2}}{4\pi^2\Gamma(\lambda'-1)\Gamma(\lambda''-1)}x^{\lambda'-1}y^{\lambda''-1}e^{-i(x,y).(\overline{w_1},\overline{w_2})}.
\end{equation}
\end{lemme}

\begin{proof}
Consider the function $g(t)=\frac{2^{\lambda-1}}{2\pi\Gamma(\lambda-1)}t^{\lambda-1}e^{-it\bar{z}}$. Then, the usual Laplace transform gives:
\begin{align*}
\mathcal{F}g(z)=\frac{2^{\lambda-1}}{2\pi\Gamma(\lambda-1)}\int_0^\infty t^{\lambda-1}e^{it(w-\bar{z})}~dt=\frac{\lambda-1}{4\pi}\left(\frac{2i}{w-\bar{z}}\right)^\lambda=K_\lambda(w,z),
\end{align*}
which proves the first point. The second is obvious from the definitions.
\end{proof}

\section{Computation of the relative reproducing kernel}
The goal of this section is to give two different proofs of Theorem \ref{thm:Fomule_adjoint_RC_relative_reproducing}. For this, we use Lemma 3.13 from \cite{KP}:

 \begin{lemme}\label{lem:calc_adj}
 Let $D_j (j=1,2)$ be some complex manifolds, and $H_j$ some Hilbert spaces of holomorphic functions on $D_j$ with reproducing kernels $K^{(j)}(\cdot,\cdot)$. If $R\ :\ H_1 \to \ H_2$ is a continous linear map, then :
 \begin{enumerate}
 \item $\overline{RK^{(1)}(\cdot,\zeta)(\tau')}=(R^*K^{(2)}(\cdot,\tau'))(\zeta) $ pour $\zeta \in D_1,\tau'\in D_2$.
 \item $(R^*g)(\zeta)=(g,RK^{(1)}(\cdot,\zeta))_{H_2}$ for $g\in H_2,\ \zeta \in D_1$.
 \end{enumerate}
\end{lemme}

Thanks to this lemma, we only need to compute $\RC \left(K_{\lambda',\lambda''}(\cdot,(w_1,w_2))\right)$ in order to find an explicit formula for the holographic operator. We give two different ways for this computation. First, we compute it directly and in a second time we use the Laplace transform to get our result.

\subsection{First proof of Theorem \ref{thm:Fomule_adjoint_RC_relative_reproducing}}\label{sec:Premiere_preuve}
Our first proof of Theorem \ref{thm:Fomule_adjoint_RC_relative_reproducing} reduces to a direct computation. 

\begin{lemme}
For $l\in \N$ and $\lambda >1$, we have:
\begin{equation}\label{eq:Partial_der_kernel}
\frac{\partial^l}{\partial z^l}K_{\lambda}(z,w)=\frac{(-1)^l(\lambda-1)_{l+1}}{4\pi (2i)^l}\left(\frac{z-\bar{w}}{2i}\right)^{-(\lambda+l)}.
\end{equation}
\end{lemme}

This statement can be proved by induction on $l$. Then we are ready to prove Theorem \ref{thm:Fomule_adjoint_RC_relative_reproducing}.
\begin{proof}[Proof of Theorem \ref{thm:Fomule_adjoint_RC_relative_reproducing}]
\begin{align*}
&R_{\lambda',\lambda''}^{\lambda'''}K_{\lambda',\lambda''}(\cdot,(w_1,w_2))(z_1,z_2)\\
&=\sum_{j=0}^l \frac{(-1)^j(\lambda'+l-j)_j(\lambda''+j)_{l-j}}{j!(l-j)!}\frac{\partial^l}{\partial z_1^{l-j} \partial z_2^j}K_{\lambda',\lambda''}(\cdot,(w_1,w_2))(z_1,z_2)\\
&=\frac{(-1)^l}{(4\pi)^2l!(2i)^l}\left(\frac{z_1-\overline{w_1}}{2i}\right)^{-(\lambda'+l)}\left(\frac{z_2-\overline{w_2}}{2i}\right)^{-(\lambda''+l)}\\
&\times \sum_{j=0}^l(-1)^j(\lambda'+l-j)_j(\lambda''+j)_{l-j}(\lambda'-1)_{l-j+1}(\lambda''-1)_{j+1}\binom{l}{j}\left(\frac{z_2-\overline{w_2}}{2i}\right)^{l-j}\left(\frac{z_1-\overline{w_1}}{2i}\right)^{j}\\
&=\frac{(-1)^l(\lambda'-1)_{l+1}(\lambda''-1)_{l+1}}{(4\pi)^2 l!(2i)^{2l}}\left(\frac{z_1-\overline{w_1}}{2i}\right)^{-(\lambda'+l)}\left(\frac{z_2-\overline{w_2}}{2i}\right)^{-(\lambda''+l)}\\
&\times \sum_{j=0}^l \binom{l}{j}(\overline{w_1}-z_1)^j (z_2-\overline{w_2})^{l-j}\\
&=\frac{(\lambda'-1)_{l+1}(\lambda''-1)_{l+1}}{(4\pi)^2 l!2^{2l}}\left(\frac{z_1-\overline{w_1}}{2i}\right)^{-(\lambda'+l)}\left(\frac{z_2-\overline{w_2}}{2i}\right)^{-(\lambda''+l)}(z_1-z_2+\overline{w_2-w_1})^l.
\end{align*}
Finally, we get by restriction to the diagonal $z_1=z_2=z$:
\begin{align*}
&\overline{\RC K_{\lambda',\lambda''}(\cdot,(w_1,w_2))(z)}=\overline{Rest|_{z_1=z_2=z}\circ R_{\lambda',\lambda''}^{\lambda'''}\left(K_{\lambda',\lambda''}(\cdot,(w_1,w_2))\right)(z)}\\
&=\frac{(\lambda'-1)_{l+1}(\lambda''-1)_{l+1}}{(4\pi)^2 l!2^{2l}}\left(\frac{w_1-\overline{z}}{2i}\right)^{-(\lambda'+l)}\left(\frac{w_2-\overline{z}}{2i}\right)^{-(\lambda''+l)}(w_2-w_1)^l.
\end{align*}
Thus Lemma \ref{lem:calc_adj} implies the statement.
\end{proof}

\subsection{Second proof of Theorem \ref{thm:Fomule_adjoint_RC_relative_reproducing}}\label{sec:Deuxieme_preuve}

Here we give a second proof of Theorem \ref{thm:Fomule_adjoint_RC_relative_reproducing} using the Laplace transform (\ref{eq:fourier_transform_SL2}).

\begin{prop}\label{prop:Pesudo_diff_kernel}
Set $\lambda',\lambda'',\lambda'''$ satisfying \eqref{condition}.
\begin{align}\label{eq:Relative_kernel_pseudodiff}
&\overline{\RC K_{\lambda',\lambda''}(\cdot,(w_1,w_2))(z)}=\\ 
&C'(w_2-w_1)^l\int_0^\infty \mathcal{F}^{-1}K_{\lambda'''}(\cdot,z)(t)~_1F_1(\lambda'+l,\lambda'+\lambda''+2l;-i(w_2-w_1)t)e^{itw_2}~dt, \nonumber
\end{align}
where $C'=\frac{\Gamma(\lambda'+\lambda''+2l-1)B(\lambda'+l,\lambda''+l)}{2^{2l+2}\pi l!\Gamma(\lambda'-1)\Gamma(\lambda''-1)}$ and $_1F_1$ is the Kummer function (see Appendix \ref{sec:hypergeo}).\\
\end{prop}

\begin{proof}
Lemmas \ref{cor:Inv_Fourier_SL2} and \ref{lem:Fourier_poids_Kummer}, and formula (\ref{eq:RC_hat}) imply:
\begin{align*}
&\widehat{\RC}\circ \mathcal{F}_2^{-1}\left(K_{\lambda',\lambda''}(\cdot,(w_1,w_2))\right)(t)\\
&=\frac{t^{\lambda'+\lambda''+l-1}}{2^3i^l\pi^2\Gamma(\lambda'-1)\Gamma(\lambda''-1)} \int_{-1}^1\Jacobi(v) e^{-it\overline{\frac{w_1+w_2}{2}+v\frac{w_2-w_1}{2}}}(1-v)^{\lambda'-1}(1+v)^{\lambda''-1}~dv\\
&=\frac{2^{\lambda'+\lambda''-4}B(\lambda'+l,\lambda''+l)}{\pi^2\Gamma(\lambda'-1)\Gamma(\lambda''-1)l!}(\overline{w_2}-\overline{w_1})^l t^{\lambda'''-1}~_1F_1\left(\lambda'+l,\lambda'+\lambda''+2l;i(\overline{w_2}-\overline{w_1})t\right) e^{-it\overline{w_2}}.
\end{align*}
As $\RC=\mathcal{F} \circ \widehat{\RC}\circ \mathcal{F}_2^{-1}$ (see \eqref{eq:fourier_transform_SL2}), it gives:
\begin{align*}
&\overline{\RC K_{\lambda',\lambda''}(\cdot,(w_1,w_2))}(z)\\
&=C'(w_2-w_1)^l\int_{\R^+} \frac{2^{\lambda'''-1}}{2\pi\Gamma(\lambda'''-1)}t^{\lambda'''-1}e^{itz}~_1F_1\left(\lambda'+l,\lambda'+\lambda''+2l;i(w_1-w_2)t\right) e^{-itw_2}~dt.
\end{align*}
Lemma \ref{cor:Inv_Fourier_SL2} gives $\frac{2^{\lambda'''-1}}{2\pi\Gamma(\lambda'''-1)}t^{\lambda'''-1}e^{-it\bar{z}}=\mathcal{F}^{-1}K_{\lambda'''}(\cdot,z)(t)$, which ends our proof.
\end{proof}

An alternative proof of Theorem \ref{thm:Fomule_adjoint_RC_relative_reproducing} is based on the properties of Kummer special functions as follows.

\begin{proof}[Second proof of Theorem \ref{thm:Fomule_adjoint_RC_relative_reproducing}]
First, remark that, for every $n\in \N$, we have:
\[\left(\frac{\partial}{\partial w}\right)^nK_{\lambda'''}(w,z)=\frac{(-1)^n(\lambda'''-1)(\lambda''')_n(2i)^{\lambda'''}}{4\pi}(w-\bar{z})^{-\lambda'''-n}.\]
Then, we use the power series expansion for Kummer functions \eqref{def:hypergeo} to get: 
\begin{align*}
&\int_0^\infty \mathcal{F}^{-1}\left(K_{\lambda'''}(\cdot,z)\right)(t)e^{itw_2}~_1F_1(\lambda'+l,\lambda'+\lambda''+2l;-i(w_2-w_1)t)~dt\\
&=\sum_{n=0}^\infty \frac{(\lambda'+l)_n}{(\lambda'+\lambda''+2l)_n}\int_0^\infty \mathcal{F}^{-1}\left(K_{\lambda'''}(\cdot,z)(t)\right)e^{itw_2}(-i(w_2-w_1)t)^n e^{itw_2}~dt\\
&=\sum_{n=0}^\infty \frac{(\lambda'+l)_n}{(\lambda'+\lambda''+2l)_n} (w_2-w_1)^n \left. \left( \frac{\partial}{\partial w}\right)^n K_{\lambda'''}(w,z) \right|_{w=w_2}\\
&=\frac{\lambda'''-1}{4\pi} \sum_{n=0}^\infty (\lambda'+l)_n(-1)^n (2i)^{\lambda'''}(w_2-\bar{z})^{-\lambda'''-n}(w_2-w_1)^n\\
&=\frac{\lambda'''-1}{4\pi}\left(\frac{w_2-\bar{z}}{2i} \right)^{-\lambda'''} \sum_{n=0}^\infty (\lambda'+l)_n(-1)^n \left( \frac{w_2-w_1}{w_2-\bar{z}}\right)^n\\
&=\frac{\lambda'''-1}{4\pi}\left(\frac{w_2-\bar{z}}{2i} \right)^{-\lambda'''} \left( 1-\frac{w_2-w_1}{w_2-\bar{z}}\right)^{-\lambda'-l}\\
&=\frac{\lambda'''-1}{4\pi}\left(\frac{w_2-\bar{z}}{2i} \right)^{-\lambda''-l}\left(\frac{w_1-\bar{z}}{2i} \right)^{-\lambda'-l} .
\end{align*}
Notice that the result is valid for $\left|\frac{w_2-w_1}{w_2-\bar{z}}\right|<1$, and use the analytic continuation to extend  it for arbitrary $w_1, w_2, z \in \Pi$.
Proposition \ref{prop:Pesudo_diff_kernel}, implies the result.
Finally, for the constant $C(\lambda',\ \lambda'')$, we have:
\begin{align*}
C(\lambda',\ \lambda'')=C'\frac{(\lambda'''-1)}{4\pi l! }
=\frac{(\lambda'''-1)\Gamma(\lambda'''-1)\Gamma(\lambda'+l)\Gamma(\lambda''+l)}{2^{2l+4}\pi^2 l!\Gamma(\lambda'-1)\Gamma(\lambda''-1)\Gamma(\lambda''')}.
\end{align*}
\end{proof}

\subsection{Link with the operator $\Psi_{\lambda',\lambda''}^{\lambda'''}$}

We give another characterization of the operators $\Psi_{\lambda',\lambda''}^{\lambda'''}$ (see \eqref{eq:Operator_holographic_holomorphe}) as follows

\begin{prop}\label{prop:Link_Psi_Relative_kernel}
Suppose $\lambda',~\lambda'',~\lambda'''$ verify \eqref{condition}. Then we have the following identity:
\begin{equation}\label{eq:Link_Psi_Relative_kernel}
\Psi_{\lambda',\lambda''}^{\lambda'''}(K_{\lambda'''}(\cdot,z))(w_1,w_2)=\frac{(-1)^l(\lambda'''-1)B(\lambda'+l,\lambda''+l)}{4\pi l!}K_{\lambda',\lambda''}^{\lambda'''}(z,w_1,w_2).
\end{equation}
where $B(x,y)$ denotes the usual Euler beta function, and $K_{\lambda',\lambda''}^{\lambda'''}$ is the relative reproducing kernel \eqref{eq:Relative_reproducing} given by 
\[K_{\lambda',\lambda''}^{\lambda'''}(z,w_1,w_2)=(w_2-w_1)^l\left(\frac{w_1-\bar{z}}{2i}\right)^{-(\lambda'+l)}\left(\frac{w_2-\bar{z}}{2i}\right)^{-(\lambda''+l)}.\]
\end{prop}

\begin{proof}
Similar argument as in the proof of Proposition \ref{prop:Pesudo_diff_kernel} gives:
\begin{align*}
&\Psi_{\lambda',\lambda''}^{\lambda'''}(K_{\lambda'''}(\cdot,z))(w_1,w_2)\\
&=\frac{(w_1-w_2)^l}{2^{\lambda'''-1}l!}\int_{-1}^1 \left(\int_0^\infty \mathcal{F}^{-1}K_{\lambda'''}(\cdot,z)(t)e^{itw(v)}~dt\right)(1-v)^{\lambda'+l-1}(1+v)^{\lambda''+l-1}~dv\\
&=\frac{2^{2l+2}\pi(-1)^l\Gamma(\lambda'-1)\Gamma(\lambda''-1)}{\Gamma(\lambda'''-1)}\overline{\RC K_{\lambda',\lambda''}(\cdot,(w_1,w_2))(z)}.
\end{align*}
\end{proof}
\textit{Remark:} This proposition gives another proof of the link between $\Psi_{\lambda',\lambda''}^{\lambda'''}$ and $(\RC)^*$ (see formula \eqref{eq:link_Psi_dualRC}). 
Indeed, for every $g\in H_{\lambda'''}(\Pi)$:
\begin{align*}
&\Psi_{\lambda',\lambda''}^{\lambda'''}(g)(w_1,w_2)\\
&=\frac{(w_1-w_2)^l}{2^{\lambda'+\lambda''+2l-1}l!}\int_{-1}^1g(w(v))(1-v)^{\lambda'+l-1}(1+v)^{\lambda''+l-1}~dv\\
&=\frac{(w_1-w_2)^l}{2^{\lambda'+\lambda''+2l-1}l!}\int_{-1}^1 \int_\Pi g(z)\overline{K_{\lambda'''}(z,w(v))}y^{\lambda'''-2}(1-v)^{\lambda'+l-1}(1+v)^{\lambda''+l-1}~dxdydv\\
&=\int_\Pi g(z)\left( \int_{-1}^1 \frac{(w_1-w_2)^l}{2^{\lambda'+\lambda''+2l-1}l!}K_{\lambda'''}(w(v),z)(1-v)^{\lambda'+l-1}(1+v)^{\lambda''+l-1}dv\right) y^{\lambda'''-2}~dxdy\\
&=\int_\Pi g(z) \Psi_{\lambda',\lambda''}^{\lambda'''}(K_{\lambda'''}(\cdot,z))(w_1,w_2)~y^{\lambda'''-2}~dxdy.\\
\end{align*}

\section{Geometrical interpretation}\label{sec:Diag_comm_SL2}

In this section, we give a geometrical interpretation of the diffeomorphism $\iota$ between $\R^+\times(-1;1)$ and $\R^+\times \R^+$ introduced in \cite{KP} (equation (2.21), p.17), and describe yet another model for the tensor product $\pi_{\lambda'}\otimes\pi_{\lambda''}$.
\subsection{Cone stratification}

Let $\sigma$ be an involutive automorphism of a connected semi-simple Lie group $G$, which commutes with the Cartan involution $\theta$ of $G$. We use the same letters $\sigma$ and $\theta$ to denote their differentials. Define a $\theta$-stable subgroup of $G$ by
\begin{equation}
G^\sigma=\{g\in G~|~\sigma(g)=g\}.
\end{equation}
Consider $K$ a maximal compact subgroup of $G$ and $\mathfrak{g}=\mathfrak{k}+\mathfrak{p}$ a Cartan decomposition of $\mathfrak{g}=\operatorname{Lie(G)}$. The space $\mathfrak{p}$ is $\sigma$-stable, so we have the following diffeomorphism for the symmetric space $G/K$:
\begin{equation}\label{eq:diffeo}
G/K \simeq \mathfrak{p}=\mathfrak{p}^\sigma \oplus \mathfrak{p}^{-\sigma},
\end{equation}
where $\mathfrak{p}^{\pm \sigma}=\{X\in \mathfrak{p}~|~\sigma(X)=\pm X\}$. This diffeomorphism is explicit using the exponential map (see \cite{Helga}, Thm.1.1, p.252):
\begin{equation}
X+Y\in \mathfrak{p}^\sigma \oplus \mathfrak{p}^{-\sigma} \mapsto \exp(X+Y)K\in G/K.
\end{equation}
The automorphism $\sigma \theta$ is also involutive, so one checks that:
\begin{equation}\label{eq:diffeo_parti}
\mathfrak{p}^\sigma \simeq G^\sigma/K^\sigma \text{ and }\mathfrak{p}^{-\sigma}\simeq G^{\sigma\theta}/K^{\sigma\theta}.
\end{equation}
Finally, \eqref{eq:diffeo} and \eqref{eq:diffeo_parti} gives the following diffeomorphism for $G/K$:
\begin{equation}\label{eq:symmetric_decomp_final}
G/K\simeq G^\sigma/K^\sigma \times G^{\sigma\theta}/K^{\sigma\theta},
\end{equation}
where the diffeomorphism for the right-hand side is explicitly given by
\begin{equation}
X+Y\in \mathfrak{p}^\sigma \oplus \mathfrak{p}^{-\sigma} \mapsto (\exp(X)K^\sigma,\exp(Y)K^{\sigma\theta})\in G^\sigma/K^\sigma\times G^{\sigma\theta}/K^{\sigma\theta}.
\end{equation}

\begin{example}
Consider the case where $G=\left\lbrace \begin{pmatrix}
a&0\\0&a^{-1}
\end{pmatrix}; a\in \R^*\right\rbrace\times \left\lbrace\begin{pmatrix}
b&0\\0&b^{-1}
\end{pmatrix}; b\in \R^*\right\rbrace$ $($which is the structure group of the Jordan algebra $\R\times\R$ $)$, and $\sigma(g_1,g_2)=(g_2,g_1)$ with $g=(g_1,g_2)\in G$. Then, $G^\sigma/K^\sigma\simeq \R^+$ is diagonaly embeded in $G/K\simeq \R^+\times \R^+$, and $G^{\sigma\theta}/K^{\sigma\theta}\simeq (-1;1)$.

The following diffeomorphism $\iota$ from $\R^+\times(-1;1)$ to $\R^+\times\R^+$: 
\begin{equation}\label{eq:diffeo_SL2}
\iota:(t,v)\mapsto \iota(t,v)\equiv(\frac{t}{2}(1-v),\frac{t}{2}(1+v))
\end{equation}
is then an explicit realization of the decomposition \eqref{eq:symmetric_decomp_final}.
\end{example}
This diffeomorphism corresponds to the diagonal embedding of $\R^+$ into $\R^+\times\R^+$, and describes the stratification of the cone $\R^+\times\R^+$ by the segments orthogonal to this diagonal. In his recent preprint \cite{JL.Clerc}, J-L. Clerc gives a generalization of this construction for the diagonal embedding of a symmetric cone $\Omega$ into the direct product $\Omega\times \Omega$ (see \cite{JL.Clerc}) for an arbitrary Euclidean Jordan algebra.\\

Let $\alpha,\beta>1$, and $d\mu_{\alpha,\beta}(v):=(1-v)^{\alpha-1}(1+v)^{\beta-1}~dv$ be a measure on the segment $(-1,1)$.
We define the Hilbert spaces 
\[L^2_{\alpha,\beta}(\R^+\times(-1;1)):=L^2(\R^+\times(-1;1),t^{\alpha+\beta -1}~dtd\mu_{\alpha,\beta}(v)).\]

For a function $f$ defined on $\R^+\times \R^+$, we define a function $T_\iota (f)$ on $\R^+\times (-1,1)$ by
\begin{equation}
T_\iota (f)(t,v)= \left(\frac{t}{2}\right)^{2-\lambda'-\lambda''}(1-v)^{1-\lambda'}(1+v)^{1-\lambda''}\cdot f\circ \iota (t,v).
\end{equation}
One checks that it is a one-to-one isometry from $L^2(\R^+\times \R^+,x^{1-\lambda'}y^{1-\lambda''}~dxdy)$ to $ L^2_{\lambda',\lambda''}(\R^+\times(-1;1))$. The inverse map is given by the following formula: 
\begin{equation}
T_{\iota}^{-1}(h)(x,y)=x^{\lambda'-1}y^{\lambda''-1}\cdot h\circ \iota^{-1}(x,y).
\end{equation}
We use this operator in order to define the \emph{stratified $L^2$-model} by transfering the restriction of the outer product of holomorphic discrete series representations of $\widetilde{SL_2(\R)}$ on the space $L^2_{\lambda',\lambda''}(\R^+\times (-1;1))$. The group $\widetilde{SL_2(\R)}$ acts on  $L^2_{\lambda',\lambda''}(\R^+\times (-1;1))$ by the operators:
\begin{equation}\label{eq:stratified_model}
T_\theta \circ \widehat{\pi_{\lambda',\lambda''}(g)}\circ T_{\theta}^{-1},
\end{equation}
where $g\in \widetilde{SL_2(\R)}$ and $\widehat{\pi_{\lambda',\lambda''}(g)}=\mathcal{F}_2^{-1}\circ \pi_{\lambda'}\hat{\otimes}\pi_{\lambda''}(g,g) \circ \mathcal{F}_2$. 

Finaly, we introduce the map $\Theta$ for any function defined on $\R^+$:
\begin{equation}
\Theta(h)(t,v)= t^{-(\lambda'+\lambda''+l-1)}\frac{P_l^{\lambda'-1,\lambda''-1}(v)}{\|P_l^{\lambda'-1,\lambda''-1}\|}h(t).
\end{equation}
This is a one-to-one isometry from $L^2_{\lambda'''}(\R^+)$ to the Hilbert space 
\begin{equation}
V_l^{\lambda',\lambda''}:=L^2\left(\R^+ ,t^{\lambda'+\lambda''-1}~dt\right)\hat{\otimes}~ \C\cdot P^{\lambda'-1,\lambda''-1}_l(v),\end{equation}
whose inverse is given by:
\begin{equation}
\Theta^{-1}(f\times P_l^{\lambda'-1,\lambda''-1})(t)=\|P_l^{\lambda'-1,\lambda''-1}\|t^{\lambda'+\lambda''+l-1}f(t),
\end{equation}
where $\|P_l^{\lambda'-1,\lambda''-1}\|=\left(\frac{2^{\alpha+\beta+1}\Gamma(l+\alpha+1)\Gamma(l+\beta+1)}{l!(2l+\alpha+\beta+1)\Gamma(l+\alpha+\beta+1)}\right)^{\frac{1}{2}}$ is the norm of Jacobi polynomials in the Hilbert space $L^2\left((-1;1),d\mu_{\lambda',\lambda''}(v)\right)$.

Analogously to \eqref{eq:stratified_model}, we use this map to transfer the holomorphic discrete series representation of $\widetilde{SL_2(\R)}$ and define, for every $g\in \widetilde{SL_2(\R)}$, the following operator on $V_l^{\lambda',\lambda''}$:
\begin{equation}
\Theta \circ \widehat{\pi_{\lambda'''}(g)}\circ \Theta^{-1}.
\end{equation}
where $\widehat{\pi_{\lambda'''}(g)}=\mathcal{F}^{-1}\circ \pi_{\lambda'''}(g) \circ \mathcal{F}$.

\subsection{Symmetry breaking and holographic operator}
According to the orthogonality of Jacobi polynomials, we have the following isomorphisms of Hilbert spaces:
\begin{align}
&L^2_{\lambda',\lambda''}(\R^+\times(-1;1))\nonumber\\
\simeq~~ & L^2\left(\R^+ ,t^{\lambda'+\lambda''-1}~dt\right)\hat{\otimes}L^2\left((-1;1),d\mu_{\lambda',\lambda''}(v)\right) \nonumber\\
\simeq~~& \left.\sum_{l\geq 0}\right.^\oplus L^2\left(\R^+ ,t^{\lambda'+\lambda''-1}~dt\right)\hat{\otimes}~ \C\cdot P^{\lambda'-1,\lambda''-1}_l(v)\label{eq:orthogonal_sum}\\
\simeq~~&\left.\sum_{l\geq 0}\right.^\oplus L^2\left(\R^+ ,t^{\lambda'+\lambda''-1}~dt\right)\hat{\otimes}~V_l^{\lambda',\lambda''}.\nonumber
\end{align}
where $\sum^\oplus$ stands for an orthogonal Hilbert sum. The last isomorphism is due to the fact that the Jacobi polynomials form a Hilbert basis for $L^2((-1;1),d\mu_{\lambda',\lambda''}(v))$ if $\lambda',~\lambda''>1$. 

This allows us to consider the orthogonal projection $\mathcal{J}^{\lambda',\lambda''}_l$ from $L^2_{\lambda',\lambda''}(\R^+\times(-1;1))$ on the Hilbert subspace $V_l^{\lambda',\lambda''}$ 
which corresponds to the usual Jacobi transform:
\begin{equation}
\mathcal{J}^{\lambda',\lambda''}_l(h)(t,v)=\frac{P_l^{\lambda'-1,\lambda''-1}(v)}{\|P_l^{\lambda'-1,\lambda''-1}\|^2}\int_{-1}^1 h(t,u)P_l^{\lambda'-1,\lambda''-1}(u)~d\mu_{\lambda',\lambda''}(u).
\end{equation}

We summarize the situation in the following Proposition.
\begin{prop}\label{thm:diagramm_SL2}
Let $\lambda',~\lambda'',~\lambda'''$ verify \eqref{condition}. The following diagram is commutative
 \[\xymatrix{
    L^2_{\lambda',\lambda''}(\R^+\times\R^+) \ar[rrr]^{c_1\widehat{RC_{\lambda',\lambda''}^{\lambda'''}} }\ar[d]_{T_\iota}  &&& L^2_{\lambda'''} (\R^+)\ar[d]^{\Theta} \ar[rr]^{c_2\Phi_{\lambda',\lambda''}^{\lambda'''}} && L^2_{\lambda',\lambda''}(\R^+\times\R^+)\\
    L^2_{\lambda',\lambda''}(\R^+\times(-1;1)) \ar[rrr]_{\mathcal{J}^{\lambda',\lambda''}_l} &&& V_l^{\lambda',\lambda''}\ar@{^{(}->} [urr]_{T_{\iota}^{-1}}
  }\]
  where $c_1=\frac{2^{\lambda'+\lambda''-3}i^l}{\|P_l^{\lambda'-1,\lambda''-1}\|}$ and $c_2=\frac{1}{\|P_l^{\lambda'-1,\lambda''-1}\|}$.
\end{prop}
\begin{proof}
Let $f\in L^2_{\lambda',\lambda''}(\R^+\times\R^+)$. On one hand side we have:
\begin{align*}
&\Phi_{\lambda',\lambda''}^{\lambda'''}\circ \widehat{\RC}f(x,y)\\
&=\frac{x^{\lambda'-1}y^{\lambda''-1}}{2i^l(x+y)^{\lambda'+\lambda''-2}}P_l^{\lambda'-1,\lambda''-1}\left(\frac{y-x}{x+y}\right)\int_{-1}^1P_l^{\lambda'-1,\lambda''-1}(u)F(\iota(x+y),u)~du.
\end{align*}
On the other hand side we have:
\begin{align*}
&T_{\iota}^{-1}\circ \mathcal{J}^{\lambda',\lambda''}_l \circ T_\iota (f)(x,y)\\
&=T_{\iota}^{-1}\left( (t,v)\mapsto \left(\frac{t}{2}\right)^{2-\lambda'-\lambda''}P_l^{\lambda'-1,\lambda''-1}(v)\int_{-1}^1f(\iota(t,u))P_l^{\lambda'-1,\lambda''-1}(u)~du \right)\\
&=\frac{2^{\lambda'+\lambda''-2}}{\|P_l^{\lambda'-1,\lambda''-1}\|^2}\frac{x^{\lambda'-1}y^{\lambda''-1}}{(x+y)^{\lambda'+\lambda''-2}}P_l^{\lambda'-1,\lambda''-1}\left(\frac{y-x}{x+y}\right)\int_{-1}^1P_l^{\lambda'-1,\lambda''-1}(u)F(\iota(x+y),u)~du.
\end{align*}
A direct computation shows that $T_{\iota}^{-1}\circ \Theta=c_2\Phi_{\lambda',\lambda''}^{\lambda'''}$.
\end{proof}

The fact that $\widehat{RC_{\lambda',\lambda''}^{\lambda'''}}$ and $\Phi_{\lambda',\lambda''}^{\lambda'''}$ are intertwinning operators (see \cite{KP}, Thm.2.11), and Proposition \ref{thm:diagramm_SL2} lead to the following:
\begin{prop}\label{fact:sym_break}
The orthogonal projection $\mathcal{J}^{\lambda',\lambda''}_l$ is a symmetry breaking operator for the $\widetilde{SL_2(\R)}$ action \eqref{eq:stratified_model} in the stratified $L^2$-model and the canonical injection corresponds to the associated holographic operator.
\end{prop}

In \cite{KP}, the authors explore the link between the symmetry breaking operator and Jacobi transform in this geometric setting (see Rmk.2.15). Proposition \ref{fact:sym_break} directly relates the $L^2$-space associated with Jacobi polynomials to the construction of symmetry breaking operators for the holomorphic discrete series representations of $\widetilde{SL_2(\R)}$ and explains the structure of the corresponding branching rules. 

\textit{Remark:} We should mention that we have similar result in the conformal case $(G,G')=(SO_0(2,n),SO_0(2,n-1))$ (see \cite{KP} for more detatils). For this we use the stratification of the time-like cone $\Omega(n)$ given by the diffeomorphism  
\begin{equation}\label{eq:diffeo_conforme}
\begin{aligned}[c]
\iota : &~ \Omega(n-1)\times(-1,1) &\to & \qquad \quad\Omega(n)\\
&\qquad\quad(y',v)&\mapsto& \left(y',-vQ_{1,n-2}(y')^{\frac{1}{2}}\right).
\end{aligned}
\end{equation}

\section{Appendix: a lemma about Kummer function}\label{sec:hypergeo}
The Kummer function, denoted $_1F_1$, is a special case of generalized hypergeometric functions defined for complex parameter $a,b$ where $b$ is not a negative integer by 
\begin{equation}\label{def:hypergeo}
_1F_1(a,b;x)=\sum_{n=0}^\infty\frac{(a)_n}{(b)_n}\frac{x^n}{n!}.
\end{equation}
For more results about Kummer function, we refer to \cite{Beals_Wong}, for example.
The Kummer function admits the following integral representation in terms of Jacobi polynomials:
\begin{lemme}\label{lem:Fourier_poids_Kummer}
Let $\alpha,\beta>1$, $l\in\N$ and $x\in \R^+$. Then:
\begin{equation}\label{eq:integral_representation_Kummer}
C_{\alpha,\beta}x^{l}e^{ix}~_1F_1(\alpha+l,\alpha+\beta+2l;-2ix)=\int_{-1}^1P_l^{\alpha-1,\beta-1}(v)e^{ivx}(1-v)^{\alpha-1}(1+v)^{\beta-1}~dv,
\end{equation}
where $C_{\alpha,\beta}=\frac{2^{\alpha+\beta+l-1}i^l}{l!}B(\alpha+l,\beta+l)$, and $B(x,y)$ is the usual Euler beta function.
\end{lemme}

\begin{proof}
The following integral representation for the Kummer function, for $Re(c)>Re(a)>0$ is classical (see \cite{Beals_Wong} p.190): 
\begin{align}\label{eq:integral_rep_Kummer_inter}
_1F_1(a,c;x)&=\frac{1}{B(a,c-a)}\int_0^1 s^{a-1}(1-s)^{c-a-1}e^{sx}~ds\nonumber\\
&=\frac{e^{\frac{x}{2}}}{B(a,c-a)2^{c-1}}\int_{-1}^1 e^{-\frac{xv}{2}}(1-v)^{a-1}(1+v)^{c-a-1}~dv,
\end{align}
where we put $v=1-2s$.

Then, the Rodrigues formula for Jacobi polynomials (see \cite{KP}, section 5.1) and integration by part give:
\begin{align*}
&\int_{-1}^1P_l^{\alpha-1,\beta-1}(v)e^{ivx}(1-v)^{\alpha-1}(1+v)^{\beta-1}~dv\\
&=\frac{(-1)^l}{2^ll!}\int_{-1}^1 e^{ixv}\left(\frac{d}{dv}\right)^l\left((1-v)^{\alpha+l-1}(1+v)^{\beta+l-1}\right)~dv\\
&=\frac{1}{2^ll!}\int_{-1}^1\left(\frac{d}{dv}\right)^le^{ixv}~(1-v)^{\alpha+l-1}(1+v)^{\beta+l-1}~dv\\
&=\frac{i^lx^l}{2^ll!}\int_{-1}^1e^{ixv}~(1-v)^{\alpha+l-1}(1+v)^{\beta+l-1}~dv.
\end{align*}
The result follows from the integral representation \eqref{eq:integral_rep_Kummer_inter} for the Kummer function with $c=\alpha+\beta+2l>\alpha+2l>0$.
\end{proof}

\bibliographystyle{alpha}
\bibliography{./Biblio}
\end{document}